\documentclass[10pt,a4paper]{amsart}
\usepackage[utf8]{inputenc}
\usepackage{amsmath,amsfonts,amssymb,amsthm}
\usepackage[colorlinks=true, pdfstartview=FitV, linkcolor=blue, citecolor=blue, urlcolor=blue,pagebackref=false]{hyperref}
\usepackage{macros}

\renewcommand*{\tilde}[1]{\widetilde{#1}}
\renewcommand*{\bar}[1]{\overline{#1}}

\newcommand*{\Exp}[1]{\mathbb{E}\left[ #1 \right]}

\newcommand{\inv}[1]{\frac{1}{#1}}
\newcommand{\ind}[1]{\textbf{1}_{\{#1\}}}

\newcommand{\rhs}{right hand side }
\newcommand{\hol}{Hölder }

\newcommand{\grad}{\triangledown }
\newcommand{\heat}{(\partial_t-\Delta)}

\renewcommand{\href}[1]{(\hyperref[#1]{\eqref*{#1}})}
\renewcommand{\leq}{\leqslant}

\newcommand{\R}{\mathbb{R}}
\newcommand{\E}{\mathbb{E}}

\newcommand{\eps}{\epsilon}

\usepackage{color}
\definecolor{darkred}{rgb}{0.9,0.1,0.1}
\definecolor{darkgreen}{rgb}{0,0.5,0}

\numberwithin{equation}{section}
\newtheorem{thm}{Theorem}[section]
\newtheorem{cor}[thm]{Corollary}
\newtheorem{lem}[thm]{Lemma}
\newtheorem{prop}[thm]{Proposition}
\newtheorem{ass}[thm]{Assumption}
\theoremstyle{remark}

\newtheorem{rmq}[thm]{Remark}

\author{Augustin Moinat \and Hendrik Weber}

\address[Augustin Moinat]{University of Warwick, Coventry, United Kingdom}
\email{a.moinat@warwick.ac.uk}

\address[Hendrik Weber]{University of Warwick, Coventry, United Kingdom}
\email{hendrik.weber@warwick.ac.uk}

\thanks{HW is supported by the Royal Society through the University Research Fellowship UF140187.}

\title{Local bounds for stochastic reaction diffusion equations}

\begin{document}

\begin{abstract}
 We prove a priori bounds for solutions of stochastic reaction diffusion equations with super-linear damping in the reaction term. 
 These bounds provide a control on the supremum of solutions on any compact space-time set which only depends on the specific 
 realisation of the noise on a slightly larger set and which holds uniformly over all possible space-time boundary values. This constitutes
 a space-time version of the so-called ''coming down from infinity" property. 

Bounds of this type are very useful to control the large scale behaviour of solutions effectively and can be used, for example, to construct 
solutions on the full space even if the driving noise term has no  decay at infinity.  

Our method shows the interplay of the large scale behaviour, dictated by the non-linearity, and the small scale oscillations, dictated by the rough driving noise. 
As a by-product we show that there is a close relation between the regularity of the driving noise term and the integrability of solutions. 
\end{abstract}
\maketitle
\section{Introduction}

We are interested in  reaction diffusion equations of the type
\begin{equation}\label{phi41}
(\partial_t-\Delta)u=-f(u)+\zeta,
\end{equation}
over $\R_t \times \R^d_x$ where $\zeta$ is an irregular distribution. 
The example we have in mind is the case where $\zeta$ is a random noise term, such as space-time 
white noise for $d=1$, or a noise which is ``white in time and coloured in space'' for $d\geq 2$.
However, we mention right away that our main result is purely deterministic and the only information 
about $\zeta$ that enters is its regularity measured in a suitable space of distributions. The non-linearity $f$ is assumed to be continuous, with super-linear growth at infinity in $u$. We derive a priori bounds on $u$. 

\medskip

It is well-known that if $f$ satisfies the so-called Osgood condition, that is if $f$ satisfies $\int_1^\infty \frac{1}{f(u)} du <\infty$, then 
solutions of the ODE $\dot{x} = - f(x)$ ``come down from infinity in finite time'' (see \cite{OS}). This means 
 that if $x$ solves the equation over $[0,t]$, then automatically  $x(t)$ satisfies a bound
which depends on $t$, but holds uniformly over all possible choices of initial datum $x(0)>0$.  
Similar statements can be derived for reaction 
diffusion equations based on a comparison principle (see e.g. \cite[Chapter 14]{Smoller} ) and also stochastic 
reaction diffusion equations (see e.g. \cite[Theorem 6.2.3]{cerrai2001second} and \cite{C}  ). 
These bounds are powerful tools to study the long-time behaviour of solutions, both in the deterministic
and in the stochastic setting - see e.g. \cite{TW} for a construction of invariant measures for 
stochastic PDEs based on such bounds.

\medskip

Our main result, Theorem~\ref{The Theorem}, is a space-time version of such a bound for 
solutions of~\eqref{phi41} with $f(u)=u|u|^{m-1}+g(u)$ where $g$ is bounded. We consider a continuous functions $u \colon \R \times \R^d \to \R$ and we assume 
that \eqref{phi41} holds for $(t,x)$ in  a cylinder\footnote{Of course the 
equation \eqref{phi41} has to be interpreted in a distributional sense, so this condition means that it holds when tested against
smooth test-functions which are supported in the cylinder, see \eqref{phi41bis}.}, say 
\begin{equation*}
P_0 :=  (0,1) \times (-1,1)^d.
\end{equation*}
Then for $R< \frac12$ 
the $L^{\infty}$ norm $\|u\|_{P_R}$ of $u$ on a smaller cylinder 
\begin{equation}\label{defpr}
P_R :=   (R^2,1) \times (-(1-R), 1-R)^d,
\end{equation} 
 satisfies a bound which only depends
on $R$ and a distributional norm of $\zeta$ restricted to the original cylinder $P_0$:
\begin{align}\label{pre-main}
\|u \|_{P_R}\leqslant C( \alpha,d,m) \max\Big\{  R^{-\frac{2}{m-1}}, [\zeta]_{\alpha-2,P_0}^{\frac{2}{2+(m-1)\alpha}} ,\|g\|^\frac1m\Big\},
\end{align}
where $[\zeta]_{\alpha-2,P_0}$  the space-time H\"older norm of order $\alpha-2$ on $P_0$ (see \eqref{shauder ou} below for a precise definition),
and $\|g\|$ refers to the supremum norm of $g$.

\medskip

One possible application of the bound \eqref{pre-main}  is the \emph{construction} of 
solutions to  \eqref{phi41} on the full space.  The standard approach to solve stochastic reaction 
diffusion equations \cite{walsh1986introduction,hairer2009introduction,da2014stochastic} consists of writing the equation in its mild form and 
solving the corresponding fixed point problem using Picard iterations. However, this approach requires a pathwise uniform-in-$x$ control on 
 $\zeta$, which typically only holds on bounded domains or  if  $\zeta$ decays at $\infty$; the interesting case of spatially stationary noise cannot 
 be treated directly in this way.  This problem was overcome in \cite{I} where solutions
 where first constructed on a sequence of growing tori and then a compactness argument in a space with weights was used to pass to the limit.
 The strong localisation obtained in \eqref{pre-main} should allow to significantly simplify this construction.

\medskip

The estimate \eqref{pre-main} also has an interesting consequence for the integrability of $u$. In fact, we are mostly interested in the case
where $\zeta$ is a random distribution with Gaussian tails such that $\E\big[ e^{ \eps [\zeta]_{\alpha-2,P_0}^2} \big]$ is finite for $\eps>0$ 
 small enough. The estimate
\eqref{pre-main} then immediately implies that for any $R>0$ and for $\eps>0$ small enough we get $\E\big[ e^{ \eps \|u \|_{P_R}^{2+(m-1)\alpha}} \big]<\infty$.
  So $\|u \|_{P_R}$ has lighter tails than Gaussian.
  We observe that better pathwise \emph{regularity} for $\zeta$ leads to better \emph{integrability} with respect to the probability distribution for $u$.
In the special case of one-dimensional reaction-diffusion equations where $\zeta$ is a space-time white noise, equation \eqref{phi41} equipped with suitable boundary 
conditions defines a reversible Markov process, and an explicit expression of the equilibrium measure is available. In Section~\ref{sec:Optimality} we argue that in this case
the integrability we derive from estimate~\eqref{pre-main} coincides with the integrability derived from the explicit invariant measure.

\medskip

Finally, our method offers a new perspective on singular SPDE. Our starting point is Hairer's notion \cite{hairer2014theory} of \emph{subcriticality} which in the context of \eqref{phi41} states, roughly 
speaking, that the \emph{small scale} behaviour of solutions should be determined by the interplay of the heat operator and the rough driving noise $\zeta$, while on \emph{large scales} the non-linearity
becomes dominant. We implement this philosophy by regularising \eqref{phi41} on a  scale $T$ by convolving the equation with a suitable regularising kernel, arriving at 
\begin{equation}\label{phi41T}
(\partial_t-\Delta)u_T=-u_T|u_T|^{m-1}+g(u)_T+\zeta_T  + [u_T|u_T|^{m-1} - (u|u|^{m-1})_T]  ,
\end{equation}
where the subscript $T$ denotes a regularised quantity. The extra term $[u_T|u_T|^{m-1} - (u|u|^{m-1})_T]  $ on the right hand side appears because regularisation and application of the polynomial do not commute. We then use a low regularity version of classical Schauder theory, Lemma~\ref{lemshauder},
to control the error term $[u_T|u_T|^{m-1} - (u|u|^{m-1})_T]  $. Using this bound, the remaining terms can be treated as in the smooth case (see Theorem~\ref{main thm poly}).

\medskip

The theory of regularity structures is indeed a main motivation for this work. A priori including the "coming-down from infinitiy" property have been proven for singular 
SPDEs, namely the dynamic $\phi^{2m}_2$ \cite{mourrat2017global,TW}  and $\phi^4_3$ models \cite{MW,albeverio2017invariant,GH}
both on  compact domains and  on the full space.  The works on $\phi^4_3$ all  relied on Fourier methods,
the method of paracontrolled distributions, rather than the theory of regularity structures. The bounds obtained there imply coming down from infinity in time only, in the case of $\phi^4_3$ on the full space \cite{GH} in a weighted space.
The ideas presented here can be extended to these more singular equations when the 
low regularity Schauder estimate,  Lemma~\ref{lemshauder}, is  replaced by a suitable version of the Schauder estimates from the theory of regularity structures.
This is the content of our companion paper \cite{MW3D}. There we show that our method significantly simplifes the technical arguments used in  \cite{MW,albeverio2017invariant,GH} 
and extend its scope to construct solutions on the full space without the need for weights. 

\medskip

In the more regular case presented here it would be natural to aim  to also include  more general non-linearities,
such as functions with faster than polynomial growth (e.g. the $\exp(\phi)$ model $f(u)=\sinh(u)$, see \cite{AKR}) or functions of slower than polynomial growth such as $f(u)\sim u \log(u)^\delta$ for $\delta>2$. 
In this case the commutator term arising in \eqref{phi41T} turns into
\[f(u_T)-f(u)_T.\]
Unfortunately, our method crucially on the fact that $xf'(x)\lesssim f(x)$ which holds for polynomial $f$, but not
for functions with exponential growth. 
Also, another part of our argument excludes functions that grow to slowly
(in the proof of Theorem~\ref{The Theorem} we need to sum $\Theta(u):= \frac{f(u)}{u}$  for $u=2^{-k}$, $k \in \mathbb{N}$),
thus essentially restricting us to polynomial $f$.
However, in the case of a more general non-linearity $f$, we implement a more standard argument based on subtracting the solution $w$ to the linear equation
\[
(\partial_t - \Delta) w  =\zeta,
\]  
and we do not pass through the regularised equation \eqref{phi41T}.
  We then get the property of ``coming down from infinity'' for the  remainder $u-w$ in Corollary~\ref{cor:dpd}. For example, when $f(u)=\sinh(u)$, the strong damping implies that 
  $u$ ``comes down from infinity" much more quickly than in the polynomial case -- in this case the function $ R^{-\frac{2}{m-1}}$ in \eqref{pre-main} turns into  $\Theta^{-1}(R^{-2})$, 
  where $\Theta(R) =\frac{\sinh(R)}{R}$. For very weak damping $f(u) \sim u\log(u)^{\alpha+2}$, we obtain a slow coming down from infinity, of order $\exp(R^{-\frac{2}{\alpha}})$.  In fact, this method is even 
 easier than the method we use for the polynomial case, but it has two significant disadvantages: On the one hand, it is impossible to measure the fine interplay between regularity of $\zeta$ and integrability of $u$ in this way, 
 because the remainder $u-w$ can never have better integrability than the Gaussian process $w$. More importantly, 
 the more sophisticated method we use in the proof of our main theorem is crucial when dealing with more singular equations \cite{MW3D}.

%
%

\medskip

The rest of the paper is structured as follows: In Section~\ref{sec:ODE}  we discuss the elementary case of 
the stochastic ODE 
\begin{equation*}
dx(t)= -|x(t)|^{m-1}x(t) dt + dw(t),
\end{equation*}
in which our strategy and also the interplay between the regularity of the noise  and integrability of the solution becomes apparent
in a technically simple context.
In Section~\ref{sec:Main} we introduce the framework and state the main result. The proof is split into Sections~\ref{sec:Caccio}--\ref{sec:proof}:
In Section~\ref{sec:Caccio} we present a proof of the ``space-time coming down from infinity'' in the case where $\zeta$
is replaced by a smooth function. The argument relies on a maximum principle.
As a Corollary, as discussed above, we derive the bounds on the remainder $u-w$
in the case of general, not necessarily polynomial $f$.
%
In Section~\ref{sec:proof} the result of Section~\ref{sec:Caccio} is applied to the regularised equation~\ref{phi41T} and combined with Schauder estimates to 
bound the commutator concluding the proof of our main result.
 In Section~\ref{sec:multipliative} we discuss the case of a random distribution $\zeta$  given by the time-derivative of the stochastic integral $\int_0^t \sigma dW$ 
 for an adapted bounded process $\sigma=\sigma(s,x)$ and a distribution valued Wiener process $W$ with suitable (spatial) covariance operator. We show
Gaussian estimates for $[\zeta]_{\alpha-2}$ and thus better than Gaussian bounds for $u$.
Finally, in the special case of space-time white noise in one spatial 
dimension we show that the integrability obtained from our method coincides with the integrability of the process
 in equilibrium obtained from the explicit invariant measure.


\section{The ODE case}\label{sec:ODE}
Before dealing with equation~\eqref{phi41} we briefly discuss the case of a (stochastic) ordinary 
differential equation
\begin{equation}\label{e:2-1}
dx(t)= - |x(t)|^{m-1} x(t) dt + dw(t)
\end{equation}
for a standard Brownian motion $w(t)$ and for $m>1$. 
It is well known that \eqref{e:2-1} defines a reversible Markov process with respect to the measure
\begin{equation}\label{e:2-2}
\mu(dx) \propto \exp\Big(- \frac{2}{m+1}|x|^{m+1}\Big) dx.
\end{equation}
We seek to derive optimal bounds on solutions of $x(t)$ directly from the equation~\eqref{e:2-1}.

\medskip

As a starting point, consider the case of an ordinary differential equation driven by a regular noise term $\eta$
\begin{equation}\label{sde}
\dot{x}(t)=-x(t)|x(t)|^{m-1}+\eta(t).
\end{equation}
A simple ODE comparison Lemma, see \cite[Lemma 3.8]{TW}, shows that for $t \in (0,1]$ 
\[
|x(t)| \leqslant C(m) \max\Big{\{}t^{-\inv{m-1}},\Big{(}\sup_{t\in[0,1]}|\eta(t)|\Big{)}^\inv{m}\Big{\}},
\]
uniformly over all choices of initial datum $x(0)$. If $\eta$ is a Gaussian process, such that the random variable $\sup_{t\in[0,1]} |\eta(t)|$ 
has finite Gaussian moments, this bound implies that for $\eps>0$ small enough
\begin{equation}\label{e:2-3a}
\Exp{\exp(\eps  |x(1)|^{2m})}<\infty.
\end{equation}
In particular,  in this regular case we get much better integrability than under the measure \eqref{e:2-2}. The following deterministic lemma shows 
that the difference in integrability  is closely related to the  regularity of the driving signal.
\begin{lem}\label{lem:ODE}
Let $w\colon [0,1] \to \R$ be $\alpha$-H\"older continuous for some $\alpha \in (0,1)$ with $w(0)=0$. For some  $m>1$ let $x\colon [0,1] \to \R$ be a continuous
solution to
\begin{equation}\label{e:2-3}
x(t) = x(0) - \int_0^t  |x(s)|^{m-1} x(s) ds + w(t) .
\end{equation}
Then for $t \in (0,1]$
\begin{equation}\label{ODE theorem}
|x(t)|\lesssim\max\Big{\{}t^{-\inv{m-1}},[w]_\alpha^{\frac{1}{1+(m-1)\alpha}}\Big{\}},
\end{equation}
where $[w]_\alpha = \sup_{0\leq s <t\leq 1} \frac{|w(t) -w(s)|}{|t-s|^\alpha}$ denotes the $\alpha$-H\"older semi-norm.
Here and in the proof we use the symbol $\lesssim$ for $\leq C(\alpha,m)$.
\end{lem}
If $w$ is a random function for which $[w]_\alpha$ has Gaussian tails this estimate yields
\begin{align*}
\Exp{ \exp\Big(\eps |x(1)|^{2+(m-1)2\alpha} \Big)} <\infty,
\end{align*}
for $\eps$ small enough. In the Brownian case where  $\alpha=\frac12-$ the exponent $2+(m-1)2\alpha$ becomes $1+m -$ in line with \eqref{e:2-2}
and as $\alpha$ approaches $1$, the exponent becomes $2m$ in line with \eqref{e:2-3a}. 

\medskip
\begin{proof}[Proof of Lemma~\ref{lem:ODE}]
As a first step, we regularise equation~\eqref{e:2-3}. To this end we introduce a smooth non-negative kernel $\Psi \colon \R \to \R$ which is compactly supported in $[0,1]$ with $\int\Psi=1$ and set  $\Psi_T(t)=\frac{1}{T}\Psi(\frac{t}{T})$. 
For any function $f\colon (0,1) \to \R$ and for $t \in (T,1)$ we define the regularisation $f_T(t) = f\ast\Psi_T(t)=\int_{t-T}^t\Psi_T(t-s) f(s) ds$.

\medskip
Convolving the integral equation~\eqref{e:2-3} with $\Psi_T$ and taking a time derivative leads to
\begin{equation}\label{sde convolve}
\dot{x}_T(t)=-x_T(t)|x_T(t)|^{m-1}+\dot{w}_T(t)+[\cdot^m,(\cdot)_T] x(t) \qquad \text{for } t \in (T,1),
\end{equation}
where we write  $[\cdot^m,(\cdot)_T] x =\left[x_T|x_T|^{m-1}-(x|x|^{m-1})_T\right]$ for the commutator term on the right hand side. 

\medskip
Now we can apply the ODE comparison Lemma  \cite[Lemma 3.8]{TW}  to get, for all $t \in (T,1]$
\begin{equation}\label{sde 1}
|x_T(t)| \lesssim\max\Big{\{}(t-T)^{-\inv{m-1}},\Big{(}\sup_{[T,1]}|\dot{w}_T|\Big{)}^\inv{m},\Big{(}\sup_{[T,1]}\big| [\cdot^m,(\cdot)_T]x \big| \Big{)}^\inv{m}\Big{\}}.
\end{equation}

\medskip
To replace $x_T$ by $x$ and to bound the commutator term on the right hand side, we need information on the regularity of $x$.
Indeed, using the fact that $\Psi$ has integral $1$, we first see for $t\in(T,1]$,
\begin{align}
|(x_T-x)(t)|=&\Big|\int_{t-T}^t \Psi_T(t-s)(x(s)-x(t))ds \Big| \notag\\
\leqslant& [x]_{\alpha,(t-T,t)}  \int_{t-T}^t \Psi_T(t-s)|s-t|^\alpha ds
\leqslant T^\alpha[x]_{\alpha,(t-T,t)},\label{sde 2.5}
\end{align}
where  $[x]_{\alpha,I}= \sup_{s \neq t \in I} \frac{|x(t) -x(s)|}{|t-s|^\alpha}$ denotes the $\alpha$-H\"older semi-norm of $x$ restricted to 
the interval $I$.
%
%
 Similarly we establish a bound on the commutator: for $s\geqslant T$, 
\begin{equation}\label{sde 3}
|[\cdot^m,(\cdot)_T ]x(s)|\lesssim \|x\|_{(s-T,s)}^{m-1} T^\alpha[x]_{\alpha,(s-T,s)},
\end{equation}
where $\|x\|_{I}$ s the supremum norm of $x$ restricted to the interval $I$. 
To see \eqref{sde 3} we first write
\[|[\cdot^m,(\cdot)_T ]x(s)|=\Big|\int_{s-T}^s\Psi_T(s-r) \big(x_T(s)|x_T(s)|^{m-1}-x(r)|x(r)|^{m-1} \big)dr\Big|.\]
Then, using the mean value theorem and $|x_T(s)|\leqslant\|x\|_{(s-T,s)}$, we have 
\[|x_T(s)|x_T(s)|^{m-1}-x(r)|x(r)|^{m-1}|\leqslant m \|x\|_{(s-T,s)}^{m-1} |x_T(s)-x(r)|.\]
Finally, using the triangle inequality in the form $|x_T(s)-x(r)|\leqslant|x_T(s)-x(s)|+|x(s)-x(r)|\leqslant 2 T^\alpha[x]_{\alpha,(s-T,s)}$,
we arrive at \eqref{sde 3}.

\medskip
In order to control the $[x]_{\alpha,(s-T,s)}$ we refer to the equation \eqref{e:2-3} once more, this time without regularisation.  
In this context a ``Schauder estimate''
is trivially derived, simply by writing for $0<t_1<t_2 <1$
\begin{align*}
|x(t_2)-x(t_1)|=&\Big|\int_{t_1}^{t_2}x(s)|x(s)|^{m-1} ds + w(t_2) - w(t_1)\Big|\\
\leqslant &|t_2-t_1|   \|x \|_{[t_1,t_2]}^m+|w(t_2)-w(t_1)|,
\end{align*}
which can be restated as
\begin{equation}\label{sde 2}
[x]_{\alpha,(s-T,s)}\leqslant T^{1-\alpha}\|x\|_{(s-T,s)}^m+[w]_\alpha.
\end{equation}
 Finally, concerning the noise term on the right hand side of \eqref{sde 1} we write
\begin{equation}\label{sde 4}
\sup_{t\in[T,1]}|\dot{w}_T(t)| = \sup_{t\in[T,1]} \Big|  \int_{t-T}^t \dot{\Psi}_T(t-s) \big( w(s) - w(t)  \big)ds  \Big|    \lesssim T^{\alpha-1}[w]_\alpha.
\end{equation}
Combining \eqref{sde 1}, \eqref{sde 3}, \eqref{sde 2} and \eqref{sde 4} we arrive at
\begin{align*}
|x_T(t) |\lesssim \max\Big{\{}&(t-T)^{-\inv{m-1}}, \Big(  T^{\alpha-1}[w]_\alpha \Big)^{\frac{1}{m}}, 
\Big(  T \|x\|_{(t-T,t)}^{2m-1}\Big)^{\frac{1}{m}} , \\
&   \Big(   T^\alpha \|x\|_{(t-T,t)}^{m-1} [w]_\alpha \Big] \Big)^{\frac1m}  \Big\}, \qquad\qquad t > T.
\end{align*}
Combining this estimate with \eqref{sde 2.5} and \eqref{sde 2}
this estimate turns into
\begin{align*}
|x(t) |\lesssim \max\Big{\{}&(t-T)^{-\inv{m-1}}, \Big(  T^{\alpha-1}[w]_\alpha \Big)^{\frac{1}{m}}, 
\Big(  T \|x\|_{(t-T,t)}^{2m-1}\Big)^{\frac{1}{m}} , \\
&   \Big(   T^\alpha \|x\|_{(t-T,t)}^{m-1} [w]_\alpha \Big] \Big)^{\frac1m}, T\|x\|^m, T^{1-\alpha} [w]_\alpha  \Big\},   \qquad\qquad t > T.
\end{align*}
Now we  choose $T=\frac{\mu}{\|x\|_{(0,1)}^{m-1}}$  for  $\mu = \mu(\alpha,m)>0$ small enough
 and consider $t$ satisfying $(t-T)^{-\frac1{m-1}}\leqslant\frac12\|x\|_{(0,1)}$. Then applying the elementary inequality $xy \leq \delta x^{p}+ C(\delta) y^{p'}$
 for $\delta>0$ and $p,p' \in (0,1)$ with $\frac{1}{p}+\frac{1}{p'} =1$ multiple times  yields
\begin{equation}\label{sde 5}
\|x(s)\|_{(2^{m-1}+\mu)\|x\|_{(0,1)}^{1-m},1)}\leqslant\max\Big{\{}\inv{2}\|x\|_{(0,1)},C[w]_\alpha^{\frac{1}{1+(m-1)\alpha}}\Big{\}},
\end{equation}
for some constant $C =C(\alpha, m)$. 
Note that we can assume that $(2^{m-1}+\mu)\|x\|_{(0,1)}^{1-m} < 1$, because else we 
trivially have a bound on $\|x\|_{(0,1)}$.

\medskip
We now define a finite set $0=t_0< \ldots  <t_N =1$ by setting   $t_{n+1}-t_n=(2^{m-1}+\mu)\|x\|^{1-m}_{(t_n,1)}$ 
as long as the times $t_{n+1}$ defined this way stay strictly less than $1$. We terminate the sequence, once this algorithm would produce 
a $t_{n+1} \geq 1$ in which case we set $t_{n+1} = t_N =1$. Note that $(2^{m-1}+\mu)\|x\|^{1-m}_{(t_n,1)}$  is increasing in $n$ so the sequence necessarily 
terminates after finitely many steps.

\medskip
Applying \eqref{sde 5} to the equation restarted at the times $t_n$ we obtain for $n \leq N-1$
\begin{equation}\label{sde 6}
\|x(s) \|_{(t_{n+1},1)}\leqslant\max\Big{\{}\inv{2}\|x\|_{(t_n,1)},C[w]_\alpha^{\frac{1}{1+(m-1)\alpha}}\Big{\}}.
\end{equation}
We now show that the estimate \eqref{ODE theorem} holds for the points $t_n$ for $n < N$.
When the maximum in \eqref{sde 6} is realised by $C[w]_\alpha^{\frac{1}{1+(m-1)\alpha}}$, then this is obvious.
Else,  we have for $k\leqslant n$, 
  $\|x\|_{(t_n,1)}\leqslant \|x\|_{(t_k,1)}2^{k-n}$ and hence
\begin{align}
\notag
t_n=&\sum_{k=0}^{n-1}t_{k+1}-t_k=(2^{m-1}+\mu)\sum_{k=0}^{n-1}\|x\|^{1-m}_{(t_k,1)}\\
\label{e:ppp}
\leqslant&(2^{m-1}+\mu)\|x\|^{1-m}_{(t_n,1)}\sum_{k=0}^{n-1}2^{(n-k)(1-m)}\lesssim\|x\|^{1-m}_{(t_n,1)}.
\end{align}
For the end point $t_N$ we have  either $t_{N-1} \geqslant \frac12$ or $t_{N}-t_{N-1}\geqslant \frac12$. In the first case 
we invoke \eqref{e:ppp} for $n = N-1$ and in the second case the definition of $t_{n+1}-t_n$, in both cases yielding
the existence of a constant $C$ such that \[ 
\|x\|_{(t_{N-1},1)}^{1-m}\geqslant C\Rightarrow\|x\|_{(t_{N-1},1)}\leqslant C^\frac1{1-m},
\]
yielding in particular a bound on $x(t_N) = x(1)$.

Finally for points $t\in(t_n,t_{n+1})$, we use the definition of $t_{n+1}-t_n$:
\begin{align*}
t&\leqslant t_{n+1}= t_{n+1}-t_n+t_n\\
&\leqslant (2^{m-1}+\mu)\|x\|^{1-m}_{(t_n,1)}+t_n\lesssim\|x\|^{1-m}_{(t_n,1)}\\
&\leqslant|x(t)|^{1-m}.
\end{align*}
\end{proof}


\section{Setting and main result}\label{sec:Main}

After this short interlude, we now go back to the parabolic equation \eqref{phi41}. Throughout the rest of the paper
we will say the a continuous function $u$ satisfies \eqref{phi41} on an open set $B \subseteq \R_t \times \R^d_x$ if for 
all smooth functions $\eta$ which are supported in $B$ we have
\begin{equation}\label{phi41bis}
\int \int u (-\partial_t-\Delta) \eta =-\int  f(u) \eta +  \int \zeta \eta,
\end{equation}
where the last integral $ \int \zeta \eta$ should be interpreted as the duality pairing between a distribution and a test function.
As usual when dealing with parabolic equations, regularity  will be measured with respect to the metric
\begin{equation}\label{parmetric}
d((t,x),(\bar{t}, \bar{x}))=\max\Big\{|x-\bar{x}|,\sqrt{|t-\bar{t}|}\Big\},
\end{equation}
where $|\cdot|$ denotes the Euclidean norm on $\mathbb{R}^d$. 
We introduce the parabolic ball of center $z=(x,t)$ and radius $R$ in this metric $d$, looking only into the past:  
\begin{equation}\label{paraball}
B(z,R)=\{\bar{z}=(\bar{t},\bar{x})\in\mathbb{R}\times\mathbb{R}^d, \, d(z,\bar{z})<R,\bar{t}<t\}.
\end{equation}

Recall that $P_R$ is the cylinder at distance $R$ from $P_0$, as introduced in \eqref{defpr}. Note that for $R'<R$ we have $P_{R'}+B(0,R'-R)\subset P_R$. 

\medskip
For $\alpha \in (0,1)$, we define the \hol semi-norm $[.]_\alpha$ 
\begin{equation}\label{e:def-hol}
[u]_\alpha:=\sup_{z\neq\bar{z}\in \mathbb{R}\times\mathbb{R}^d}\frac{|u(z)-u(\bar{z})|}{d(z,\bar{z})^\alpha}.
\end{equation}
We will often deal with local quantities: If $B\subset\mathbb{R}\times\mathbb{R}^d$ is a bounded set, then we define the local $\alpha$-H\"older semi-norm  $[.]_{\alpha,B}$ as in \eqref{e:def-hol} with the supremum  restricted to $z,\bar{z}\in B$. 
Similarly, $\|.\|$ denotes the supremum norm on the whole space $\R \times \R^d$ and $\|.\|_{B}$ the supremum norm over $B$.

\medskip

To measure distributions in negative \hol spaces, we introduce a family of mollification operators $\{(.)_T\}$ which are consistent with the scaling given by the heat operator $(x,t)=(l\bar{x},l^2\bar{t})$.
For this we fix a non-negative smooth function $\Psi$ with support in $-B(0,1)$ 
with $\Psi(z)\in [0,1]$ for all $z$ and with integral $1$ 
and for $T \in (0,1]$ set
$\Psi_T(x,t)=\inv{T^{d+2}}\Psi\left(\frac{x}{T},\frac{t}{T^2}\right)$. We define the operator $(\cdot)_T$ by convolution with $\Psi_T$, noting that
for any $T$, $(\cdot)_T$ is a contraction on with respect to $\| \cdot \|$. We wish to keep track of the support of the relevant functions. Since $\Psi_T$ is compactly support in $-B(0,T)$,
\begin{align}\label{Lp scaling ball}
\|h_T\|_{C}\leqslant\|h\|_{C+B(0,T)}
\end{align}
for any bounded set $C$. 
Furthermore, we mention the estimate
\begin{equation}\label{moment of psi}
\int|\Psi_T(x-y)|d(x,y)^\alpha dy\leqslant T^\alpha,
\end{equation}
which, as in \eqref{sde 2.5} above, immediately implies that for any $h\in C^\alpha$, and for any bounded set $C$,
we have  
\begin{equation}\label{mollified regularity}
\|h_T-h\|_{C}\leqslant T^\alpha\sup_{z\in C}[h]_{\alpha,B(z,T)}.
\end{equation}
Finally, we define the local $C^{\alpha-2}$ semi-norm of a distribution $\zeta$ for $\alpha-2<0$ as 
\begin{equation}\label{shauder ou}
[\zeta]_{\alpha-2,C}=\sup_{T\leq 1}\|(\zeta)_T\|_{C}T^{2-\alpha}.
\end{equation}
This is a localised version of the Besov norm of $B^{\alpha-2}_{\infty,\infty}$ as defined, for example in 
\cite[Theorem 2.34]{bahouri2011fourier}. Note that, $[\zeta]_{\alpha-2,C}$ depends only on the behaviour of the 
 distribution $\zeta$ on the set $C+B(0,1)$ (i.e. if $\zeta$ and $\tilde \zeta$ coincide when tested against 
 test-functions supported in this set, then $ [\zeta- \tilde \zeta]_{\alpha-2,C}=0$). Multiplication with a smooth 
 function is a continuous operation with respect to this norm. We have for any smooth and compactly supported function
 $\eta$
 \begin{align}\label{e:etazeta}
 [ \eta \zeta ]_{\alpha-2} \leq C(\eta) [\zeta]_{\alpha-2, \mathrm{supp}(\eta)}.
 \end{align}
 Estimates of this type are classical and are typically proved by choosing a convenient mollifying kernel $\Psi_T$,
 see e.g. \cite{OW} for estimates based on kernels $\Psi_T$ satisfying a semi-group property in $T$, or 
 \cite[Section 2.4]{bahouri2011fourier} for a  proof in the language of Littlewood-Paley theory. We refer to 
 \cite[Lemma A3]{OW} for a proof that norms defined for different kernels are equivalent.
More complicated bounds of this type are also essential in our companion paper  \cite{MW3D} and are 
discussed there at length.
 
%
%
%

\medskip

We now state our main result, to be proven in Section~\ref{sec:proof}. 
\begin{thm}\label{The Theorem}
Assume that $f(u)=u|u|^{m-1}+g(u)$ with $m\geqslant1$, $g$ bounded and $\zeta$ is of regularity $\alpha-2$ for some $\alpha>0$ in the sense of \eqref{shauder ou}. 
There exists a constant $C=C(\alpha,m,d)$ such that if $u$ is continuous and solves  \eqref{phi41} on the cylinder $P_0$ in the sense of of \eqref{phi41bis}
 then for all $R\in(0,\inv{2})$,
\begin{equation}\label{The Theorem eq}
\|u\|_{P_R}\leqslant C\max\left\{R^{-\frac{2}{m-1}},[\zeta]_{\alpha-2,P_0}^\inv{1+(m-1)\frac\alpha2},\|g\|^\frac1m\right\}.
\end{equation}

\end{thm}

\section{Maximum principle}\label{sec:Caccio}
\subsection{Assumptions and statement}

We prove a space-time version of ``coming down from infinity'' when there is no distribution of negative regularity involved, but we allow for a more general non-linearity. Let $u$ be a  $C^2$ function defined for $z \in\R \times\R^{d}$, for which the following holds point-wise in $P_0$: 
\begin{equation}\label{eq:max rd 1}
(\partial_t-\Delta)u(z)=-f(u,z)+g(u,z).
\end{equation}
\begin{ass}\label{ass1}
We make the following assumptions on $f$ and $g$:
\begin{enumerate}
\item $g$ is a bounded function;
\item there exists an antisymmetric  $C^2$ function which we also denote $f$ such that for all $z$, for all $u>0$, $f(u,z)\geqslant f(u)$ and $-f(-u,z)\geqslant f(u)$;
\item $f''(u)\geqslant 0$ for $u>0$;
\item there exists a constant $c>1$ such that $uf'(u)\geqslant cf(u)$.
\end{enumerate}
Define $\Theta(u)=\frac{f(u)}u$. By $(4)$, $\Theta$ is increasing.
\end{ass}

\begin{thm}\label{main thm poly 1}
Let $u\in C^\infty$ solve \eqref{eq:max rd 1} for functions $f$ and $g$ satisfying Assumption~\ref{ass1}. 
There exist $\lambda=\lambda(d)>0$  and $C=C(c,d)$ such that the following point-wise bound on $u$,  holds for all $(t,x)\in(0,1)\times (-1,1)^d$:
\begin{equation}\label{eq:max theorem 1}
|u(x,t)|\leqslant C\max\Big\{\Theta^{-1}\Big(\frac1{\lambda^2\min\{t,(1-x_i)^2,(1+x_i)^2,i=1...d\}}\Big),f^{-1}(\|g\|)\Big\}.
\end{equation}
\end{thm}
Note that $\min\{t,(1-x_i)^2,(1+x_i)^2,i=1...d\}$ is exactly the square of the distance to the boundary of $[0,1]\times [-1,1]^d$ in the parabolic metric. In other words, this gives a bound on $||u||_{ P_R}$, depending only on $R$.

%

The condition $xf'(x)\geqslant cf(x)$ with $c>1$ is verified exactly for $f(u)=u|u|^{c-1}$, hence any function with at least polynomial growth is included in this theorem. For such monomials, $\Theta^{-1}$ become $x\mapsto x^\frac1{c-1}$. For functions with faster growth, the bound is going to be even stronger. However, some functions with super-linear but not polynomial growth  are not included. For example $f(u)=u\log(1+u)^{\alpha}$ for $\alpha>0$. For this example, $\frac{uf'(u)}{f(u)}=1+\frac{u\alpha}{(1+u)\log(1+u)}\rightarrow 1$ as $u\rightarrow\infty$, so point (4) in Assumption~\ref{ass1} is violated. We can still get a result in that case, under a  different set of assumptions:

\begin{ass}\label{ass2}
We make the following assumptions on $f$ and $g$:
\begin{enumerate}
\item $g$ is a bounded function;
\item there exists an antisymmetric $C^2$ function which we also denote $f$ such that for all $z$, for all $u>0$, $f(u,z)\geqslant f(u)$ and $-f(-u,z)\geqslant f(u)$;
\item $uf'(u)\geqslant f(u)$ and there exist two $C^2$ functions $f_1$ and $f_2$ such that $f=f_1f_2$;
\item $f_1$ is antisymmetric, $f_1''\geqslant 0$ for $u>0$;
\item $f_2\geqslant c>0$ and
\begin{equation}\label{eq:max cond f1f2}
f_2(u)\geqslant\max\Big\{\frac1{\Big(\frac{uf_1'(u)}{f_1(u)}-1\Big)^2},\frac1{\frac{uf_1'(u)}{f_1(u)}-1}\Big\}.
\end{equation}
\end{enumerate}
Define now $\Theta(u)=\frac{f_1(u)}u$. $\Theta$ is increasing by condition $(5)$.
\end{ass}

In the example where we want to take $f_1(u)=u\log(1+u)^\alpha$, one can easily check that in order to satisfy condition $(5)$, $f_2$ should be $\Big(\frac{1+u}{\alpha u}\Big)^2\frac{\log(1+u)^2}\alpha$ and hence $f(u)=\frac{(1+u)^2}{\alpha^2u}\log(1+u)^{2+\alpha}$ and $\Theta^{-1}(x)=\exp(x^\frac1\alpha)-1$.

%

\begin{thm}\label{main thm poly}
Let $u\in C^\infty$ solve \eqref{eq:max rd 1} for functions $f$ and $g$ satisfying Assumption~\ref{ass2}. 
There exist $\lambda=\lambda(d)>0$  and $C=C(c,d)$ such that the following point-wise bound on $u$,  holds for all $(t,x)\in(0,1)\times (-1,1)^d$:
\begin{equation}\label{eq:max theorem}
|u(x,t)|\leqslant C\max\Big\{\Theta^{-1}\Big(\frac1{\lambda^2\min\{t,(1-x_i)^2,(1+x_i)^2,i=1...d\}}\Big),f^{-1}(\|g\|)\Big\}.
\end{equation}
\end{thm}
Theorem \ref{main thm poly 1} is implied by Theorem \ref{main thm poly} by choosing  $f_1=f$ and $f_2=\frac1{(c-1)^2}$.

\begin{rmq}
The fact that under these more general assumptions $\Theta$ is not simply defined by $f(u)/u$ but instead grows more slowly, is the 
reason why  we do not get an equivalent of Theorem~\ref{The Theorem}, in the case of slower than polynomial growth.
\end{rmq}
\subsection{Bound on the remainder}\label{subsec:Cacciodpd}

A first corollary of this result is a ``coming down from infinity'' result for the singular equation \eqref{phi41} with  general non linearity. In the manner of \cite{DPD}, we expand around the solution to the linear equation: let $w$ be the solution to  
\begin{equation}\label{eq:linear}
(\partial_t-\Delta)w=\zeta,
\end{equation}
with Dirichlet boundary conditions on $P_0$. We will show (in Lemma~\ref{lemshauder}) that $w\in L^\infty$ if $\zeta\in C^{\alpha-2}$ for $\alpha>0$. 
Then define $v=u-w$. If $u$ is a solution to
\[
\heat u(z)=-f(u,z)+g(u,z)+\zeta,
\]
 where we assume that $f,g$ satisfies the Assumption~\ref{ass2} then $v$ is a solution to 
\begin{equation}\label{eq:dpd}
\heat v(z)=-f(v+w,z)+g(v+w,z)
\end{equation}
on $P_0$.
We now use the $w$-dependent decomposition $f(v+w,z) = \tilde{f}(v,z) +\tilde{g}(v,z)$ defined by
\begin{align*}
\tilde{f}(v,z) = 
\begin{cases}
f(v+w,z)  \qquad &\text{if  } |v(z)| \geq 2 |w(z)| \\
f\Big(\frac{v(z)}{2} \Big) \qquad  &\text{else},
\end{cases}
\end{align*}
and $\tilde{g}(v,z) = f(v+w,z)-  \tilde{f}(v,z)$. Then, on the one hand by monotonicity of $f$ we have
$\tilde{f}(v,z)\geqslant f(\frac{v}{2})$ and on the other hand $\|\tilde{g}\|\leqslant f(3\|w\|)$.
 The Assumptions~\ref{ass2} are then satisfied with $\tilde{f}$ and $g+\tilde{g}$ and we can apply Theorem~\ref{main thm poly} to get a bound on $v$, and then the triangle inequality to get bounds on $u$. 
 We have
\[
f^{-1}(\|g+\tilde{g}\|)\leqslant f^{-1}(2\|g\|)+6\|w\|.
\]
 A corollary of Theorem~\ref{main thm poly} is then:
\begin{cor}\label{cor:dpd}
Assume $\zeta\in C^{\alpha-2}$ for some $\alpha>0$. If $u$ is solution to \eqref{phi41} and $w$ is solution to \eqref{eq:linear}, then there exists constants $C= C(c,d,\alpha)$ and $\lambda=\lambda(d)$ such that
\begin{equation}\label{eq:caccio dpd}
\|u\|_{P_R}\leqslant C \max\Big\{\Theta^{-1}((\lambda R)^{-2}),f^{-1}(2\|g\|),\|w\|\Big\}.
\end{equation}
\end{cor}
Keeping in mind the motivation of stochastic PDEs, where $\zeta$ is the white noise, the drawback of the expansion around the solution to the linear equation is that the integrability of $u$ that we get out of this result is at best the one of $w$. As we will see in section~\ref{sec:Optimality}, Theorem~\ref{main thm poly} allows for better estimates than this.

\subsection{Proof of Theorem~\ref{main thm poly}}
We only prove the bound for the positive part of $u$. The bound for the negative part follows by symmetry.
Let $\eta$ be a continuous function defined on $\R_+\times[-1,1]^d$, $C^2$ and strictly positive on the interior and such that $\eta=0$ on the boundary. 
Either $u\eta$ attains its maximum on $[0,1]\times [-1,1]^d$ at some point $z_0\in(0,1]\times (-1,1)^d$, or it is non-positive, in which case $u\leqslant0$ in $[0,1]\times \{|x|\leqslant1\}$. Assuming this is not the case, we get that at the maximum point, $0=\grad (u\eta)(z_0)$, i.e. 
\begin{equation}\label{eq:max grad u eta}
\grad u=-\frac{\grad\eta}{\eta}u.
\end{equation}
If $z_0\in \{1\}\times (-1,1)^d$, then $\partial_tu(z_0)\geqslant 0$. Else, $\partial_tu(z_0)= 0$. Additionally, $\Delta u(z_0)\leqslant 0$ and therefore at the maximum we have
\begin{align*}
0\leqslant& \heat(u\eta)=\eta\heat u+u\heat\eta-2\grad u.\grad \eta\\
\overset{\eqref{eq:max rd 1};\eqref{eq:max grad u eta}}{=} &-\eta (f(u,z)-g(u,z))+u\Big(\heat\eta+2\frac{|\grad\eta|^2}{\eta}\Big).
\end{align*}
Assume $\eta$ satisfies the following inequality:
\begin{equation}\label{eq:max eta}
 \frac{\heat\eta}\eta+2\frac{|\grad\eta|^2}{\eta^2}\leqslant \frac\eta2 f(\frac1\eta).
\end{equation}
Then we get
\begin{equation}\label{eq:max u eta}
 \frac{f(u)}{u}\leqslant \frac\eta2 f\Big(\frac1\eta\Big)+\frac{\|g\|}{u}\leqslant 2\max\Big{\{}\frac\eta2 f\Big(\frac1\eta\Big),\frac{\|g\|}{u}\Big{\}}.
\end{equation}
If the maximum is realised by the first term, then $\frac{f(u)}{u}\leqslant \eta f(\inv{\eta})$. Since $uf'(u)\geqslant f(u)$, $u\mapsto\frac{f(u)}u$ is increasing, we have that at $z_0$, $u\eta\leqslant 1$. If the maximum is realised by the second term, then it has to be bigger than the first one :
\[
\frac\eta2 f\Big(\frac1\eta\Big)\leqslant\frac{\|g\|}{u}\Rightarrow u\eta\leqslant2\frac{\|g\|}{f(\frac1\eta)}.
\]
We then have that at $z_0$, $u\eta\leqslant 2$ under the condition
\begin{equation}\label{condition eta}
\eta\leqslant\frac1{f^{-1}(\|g\|)}.
\end{equation}

 In both cases, we obtain that $u\leqslant \frac2\eta$ on all of $[0,1]\times \{|x|\leqslant1\}$. With a choice of $\eta$ satisfying the inequalities \eqref{eq:max eta} and \eqref{condition eta}, we obtain good bounds on the function $u$. We choose the following for $z=(t,x)\in(0,\infty)\times(-1,1)$, for some value $\lambda$ to be defined:
\begin{equation}\label{eq:max def eta}
\eta(x,t)=\frac1{\Theta^{-1}(\frac1{\lambda^2 t})+\sum_{i=1}^d\left(\Theta^{-1}(\frac1{\lambda^2(1+x_i)^2})+\Theta^{-1}(\frac1{\lambda^2(1-x_i)^2})\right)+f^{-1}(\|g\|)},
\end{equation}
and we continuously extend with the value $0$ on the boundary of the domain.
This choice of $\eta$ guarantees a bound on $u$ that is related to the distance from the boundary of $[0,1]\times[-1,1]^d$, independently of the boundary conditions. Indeed,
\begin{align}\label{eq:max comp eta}
( 2d+1)\Theta^{-1}\Big(&\frac1{\lambda^2\min_i\{ t,(1+x_i)^2,(1-x_i)^2\}}\Big)\nonumber\\
&\geqslant\frac1\eta-f^{-1}(\|g\|)\geqslant\Theta^{-1}\Big(\frac1{\lambda^2\min_i\{ t,(1+x_i)^2,(1-x_i)^2\}}\Big).
\end{align}
It also satisfies $0\leqslant\eta\leqslant\frac1{f^{-1}(\|g\|)}$.
 We will now check \eqref{eq:max eta}:
\[
\partial_t\eta=\frac{\lambda^2}{(\lambda^2 t)^2}\frac1{\Theta'\circ\Theta^{-1}(\frac1{\lambda^2 t})}\eta^2.
\]
 We use $v=\Theta^{-1}(\frac1{\lambda^2 t})\leqslant\frac1\eta$ and $f= f_1f_2$. Given that $\Theta'(y)=\frac{f_1'(y)}y-\frac{f_1(y)}{y^2}$, we get
\[
\frac{\partial_t\eta}{\eta^2f(\frac1\eta)}\leqslant\frac{\lambda^2}{f(\frac1\eta)}\frac{\Theta(v)^2}{\Theta'(v)}\leqslant \frac{\lambda^2}{f_1(v)f_2(v)}\frac{\Theta(v)^2}{\Theta'(v)}=\frac{\lambda^2}{f_2(v)}\frac1{\frac{vf_1'(v)}{f_1(v)}-1}.
\]
Applying the condition \eqref{eq:max cond f1f2} gives a bound on this, independent of $v$. We now consider the spatial derivatives.
\[
\partial_i\eta=\frac{1}{\lambda^2}\Big(\frac{2}{(1+x_i)^3}\frac1{\Theta'\circ\Theta^{-1}(\frac1{\lambda^2(1+x_i)^2})}-\frac{2}{(1-x_i)^3}\frac1{\Theta'\circ\Theta^{-1}(\frac1{\lambda^2(1-x_i)^2})}\Big)\eta^2.
\]
\begin{align*}
\partial_i^2\eta=&-\frac{1}{\lambda^2}\Big(\frac{6}{(1+x_i)^4}\frac1{\Theta'\circ\Theta^{-1}(\frac1{\lambda^2(1+x_i)^2})}+\frac{6}{(1-x_i)^4}\frac1{\Theta'\circ\Theta^{-1}(\frac1{\lambda^2(1-x_i)^2})}\Big)\eta^2\\
+& \frac{1}{\lambda^4}\Big(\frac{4}{(1+x_i)^6}\frac{\Theta''\circ\Theta^{-1}(\frac1{\lambda^2(1+x_i)^2})}{\big(\Theta'\circ\Theta^{-1}(\frac1{\lambda^2(1+x_i)^2})\big)^3}+\frac{4}{(1-x_i)^6}\frac{\Theta''\circ\Theta^{-1}(\frac1{\lambda^2(1-x_i)^2})}{\big(\Theta'\circ\Theta^{-1}(\frac1{\lambda^2(1-x_i)^2})\big)^3}\Big)\eta^2\\
+& \frac{2}{\lambda^4} \Big(\frac{2}{(1+x_i)^3}\frac1{\Theta'\circ\Theta^{-1}(\frac1{\lambda^2(1+x_i)^2})}-\frac{2}{(1-x_i)^3}\frac1{\Theta'\circ\Theta^{-1}(\frac1{\lambda^2(1-x_i)^2})}\Big)^2\eta^3.
\end{align*}
Note that the last line is equal to $\frac{\partial_i\eta}{\eta^2}2\eta\partial_i\eta=2\frac{(\partial_i\eta)^2}{\eta}$, hence it will cancel when computing $-\partial_i^2\eta+2\frac{(\partial_i\eta)^2}{\eta}.$
For the remaining terms, we use  $v_{1,i}=\Theta^{-1}(\frac1{\lambda^2(1+x_i)^2})$ and $v_{2,i}=\Theta^{-1}(\frac1{\lambda^2(1-x_i)^2})$ and we get:
\begin{align*}
\frac1{\lambda^2f(\frac1\eta)}\Big(-\frac{\partial_i^2\eta}{\eta^2}+2\frac{(\partial_i\eta)^2}{\eta^3}\Big)=&\frac{6\Theta(v_{1,i})^2}{f(\frac1\eta)(\frac{f_1'(v_{1,i})}{v_{1,i}}-\frac{f_1(v_{1,i})}{v_{1,i}^2})}+\frac{6\Theta(v_{2,i})^2}{f(\frac1\eta)(\frac{f_1'(v_{2,i})}{v_{2,i}}-\frac{f_1(v_{2,i})}{v_{2,i}^2})}\\
&-\frac{4\Theta(v_{1,i})^3}{f(\frac1\eta)}\frac{\frac{f''(v_{1,i})}{v_{1,i}}-\frac2{v_{1,i}}(\frac{f_1'(v_{1,i})}{v_{1,i}}-\frac{f_1(v_{1,i})}{v_{1,i}^2})}{(\frac{f_1'(v_{1,i})}{v_{1,i}}-\frac{f_1(v_{1,i})}{v_{1,i}^2})^3}\\
&-\frac{4\Theta(v_{2,i})^3}{f(\frac1\eta)}\frac{\frac{f''(v_{2,i})}{v_{2,i}}-\frac2{v_{2,i}}(\frac{f_1'(v_{2,i})}{v_{2,i}}-\frac{f_1(v_{2,i})}{v_{2,i}^2})}{(\frac{f_1'(v_{2,i})}{v_{2,i}}-\frac{f_1(v_{2,i})}{v_{2,i}^2})^3}.
\end{align*}
Using that $f$ is increasing, the bound \eqref{eq:max comp eta} and $f=f_1f_2$, we have that $f(\frac1\eta)\geqslant f_1(v_{j,i})f_2(v_{j,i})$ for $j\in\{1,2\}$. We also know that $f_2''>0$, hence we get: 
\begin{align*}
\frac1{f(\frac1\eta)}\Big(-\frac{\partial_i^2\eta}{\eta^2}+2\frac{(\partial_i\eta)^2}{\eta^3}\Big)
\leqslant&\frac{6\lambda^2}{f_2(v_{1,i})(\frac{f_1'(v_{1,i})v_{1,i}}{f_1(v_{1,i})}-1)}+\frac{6\lambda^2}{f_2(v_{2,i})(\frac{f_1'(v_{2,i})v_{2,i}}{f_1(v_{2,i})}-1)}\\
+&\frac{8\lambda^2}{f_2(v_{1,i})(\frac{f_1'(v_{1,i})v_{1,i}}{f_1(v_{1,i})}-1)^2}+\frac{8\lambda^2}{f_2(v_{2,i})(\frac{f_1'(v_{2,i})v_{2,i}}{f_1(v_{2,i})}-1)^2}.
\end{align*}
We conclude this proof by using the condition \eqref{eq:max cond f1f2} and setting  $\lambda=(28d+1)^{-\frac12}$. 

\section{Proof of the main result}\label{sec:proof}
\subsection{Low regularity Schauder estimate}
We give here a proof of a low regularity Schauder estimate in our setting.
\begin{lem}\label{lemshauder}
Let $u$ be compactly supported  in $B(0,1)$ and let $f:=(\partial_t-\Delta)u$. Then for $\alpha >0$ there exists a constant $C=C(\alpha,d)$ such that
\begin{equation}\label{Schauder lemma}
[u]_\alpha\leqslant C\sup_{T\leqslant 1}T^{2-\alpha}\|f_T\|.
\end{equation}
\end{lem}
\begin{proof}
Throughout the proof, $\lesssim$ will denote a bound up to a multiplicative constant, which may change from line to line, but which always depends only on $\alpha$ and $d$.
Define $N=\sup_{T\leqslant 1}T^{2-\alpha}\|f_T\|$. Since $(\cdot)_T$ denotes the convolution with a smooth kernel, it commutes with derivatives. We know that for $T<1$, for any $l\in \text{span}\{1,x_i,i\in\{1,...,d\}\}$, we have on $\mathbb{R}\times\mathbb{R}^d$,
\[(\partial_t-\Delta)(u_T-l)=f_T.
\]
For $z_0\in B(0,1)$, for some $L>0$ to be fixed below, define $v_>$ as the solution to   
\[(\partial_t-\Delta)v_>=\ind{B(z_0,L)}f_T,\quad v_>|_{\partial B(z_0,L)}=0,\]
where $\partial B(z_0,L)=\{z=(t,x),d(z,z_0)=L,t\leqslant t_0\}$ is the parabolic boundary of $B(z_0,L)$.
The first interesting inequality we get from standard heat equation estimates \cite[Cor.8.1.5]{krylov1996lectures} is
\begin{equation}\label{first inequality}
\|v_>\|\lesssim L^2\|f_T\|\leqslant  L^2T^{\alpha-2}N.
\end{equation}
Define $v_<=u_T-v_>$. As $(\partial_t-\Delta)v_<=0$ on $B(z_0,L)$ for any differential operator  $D\in\{\partial_t,\partial_i\partial_j,i,j\in\{1,...,d\}\}$,
\[\|Dv_<\|_{B_{\frac{L}{2}}}\lesssim L^{-2}\inf_l\|u_T-l\|_{B(z_0,L)},\]
where $l$ runs over all function spanned by $1$ and $x_i,i\in\{1,...,d\}$.
Therefore, for any $R<\frac{L}{2}$, for the same range of operator $D$, for a suitably chosen $l_R\in \text{span}\{1,x_i,i\in\{1,...,d\}\}$,
\[\|v_<-l_R\|_{B(z_0,R)}\leqslant R^2\|Dv_<\|_{B(z_0,R)}\lesssim\left(\frac{R}{L}\right)^2\inf_l\|u_T-l\|_{B(z_0,L)}.\]
Using the definition of $v_<$ and the triangle inequality,
\[\|u_T-l_R\|_{B(z_0,R)}-\|v_>\|_{B(z_0,R)}\lesssim\left(\frac{R}{L}\right)^2\inf_l\|u_T-l\|_{B(z_0,L)}.\]
From \eqref{first inequality}, 
\begin{align}\label{second inequality}
\inv{R^\alpha}\|u_T-l_R\|\lesssim& \left(\frac{R}{L}\right)^{2-\alpha}\inv{L^\alpha}\inf_l\|u_T-l\|_{B(z_0,L)}\nonumber\\
&+\left(\frac{L}{T}\right)^2\left(\frac{T}{R}\right)^\alpha N.
\end{align}
Furthermore, from \eqref{mollified regularity} we get,
\begin{equation}\label{third inequality}
\inv{R^\alpha}\|u-l_R\|_{B(z_0,R)}\lesssim \inv{R^\alpha}\|u_T-l_R\|_{B(z_0,R)}+\left(\frac{T}{R}\right)^\alpha[u]_{\alpha}.
\end{equation}
Similarly, for any $l\in \text{span}\{1,x_i,i\in\{1,...,d\}\}$
\begin{equation}\label{third inequality bis}
\inv{L^\alpha}\|u_T-l\|_{B(z_0,L)}\lesssim \inv{L^\alpha}\|u-l\|_{B(z_0,L)}+\left(\frac{T}{L}\right)^\alpha[u]_{\alpha}.
\end{equation}
Hence for $0<\epsilon<1$, for $T=\epsilon R=\epsilon^2L$, \eqref{second inequality} and \eqref{third inequality},\eqref{third inequality bis} give:
\begin{align}\label{fourth inequality}
\inv{R^\alpha}\inf_l\|u-l\|_{B(z_0,R)}\lesssim \epsilon^{2-\alpha}\inv{L^\alpha}\inf_l\|u-l\|_{B(z_0,L)}\\+(\epsilon^\alpha+\epsilon^{2\alpha})[u]_{\alpha}+\epsilon^{\alpha-4}N.
\end{align}
Note that $[u]_\alpha\sim\sup_{z_0}\sup_L\inv{L^\alpha}\inf_l\|u-l\|_{B(z_0,L)}$, hence
\begin{equation}\label{fifth inequality}
[u]_\alpha\lesssim(\epsilon^{2-\alpha}+\epsilon^\alpha+\epsilon^{2\alpha})[u]_\alpha+\epsilon^{\alpha-4}N.
\end{equation}
By making $\epsilon$ small enough, we can absorb $[u]_{\alpha}$ in the \rhs of \eqref{fifth inequality} into the left hand side, concluding the proof of the Schauder estimate \eqref{Schauder lemma}. Note that for this last step we needed the assumption $[u]_\alpha<\infty$. This assumption can be removed as by regularising the equation first, we have that uniformly for any $\tau>0$,
\[[u_\tau]_\alpha\lesssim\sup_{T\leqslant 1}T^{2-\alpha}\|(f_\tau)_T\|\leqslant\sup_{T\leqslant 1}T^{2-\alpha}\|f_T\|,\]
and as $u$ is continuous, we can pass to the limit for $\tau\rightarrow 0$.
\end{proof}

\subsection{Proof of Theorem~\ref{The Theorem}}

From now on, $f(u)=|u|^{m-1}u$. In particular, Theorem~\ref{main thm poly 1} holds with $\Theta^{-1}(R^{-2})=R^{-\frac1{m-1}}$. The proof relies on two arguments. The small scale oscillations are controlled via Schauder theory and the large scale behaviour through the maximum principle derived in section~\ref{sec:Caccio}, which applies only to regular objects. A connection between the two is established via the convolution of the equation with the kernel introduced in Section~\ref{sec:Main}, which produces a commutator term. The technicality of the proof lies in balancing the contribution of the commutator and the contribution of the irregular noise.

Throughout the proof, $\lesssim$ will denote a bound up to a multiplicative constant which may change from line to line, but will only depend on $d$, $m$ and $\alpha$. We will also write $u^m$ as a short-hand for $u|u|^{m-1}$, as in the case when $m$ is an odd integer.

\paragraph*{\textsc{Step} 1: Local Schauder estimate}
We claim that for any $R>0$, for any $k>2$,
\begin{align}\label{schauder}
[u]_{\alpha,B(z,R)}\lesssim &\sup_{T\leqslant kR}T^{2-\alpha}\|\heat (u\ind{B(z,kR)})_T\|+(kR)^{-\alpha}\|u\|_{ B(z,kR)}\nonumber\\
\leqslant&(kR)^{2-\alpha}\|u\|^m_{ B(z,kR)}+(kR)^{2-\alpha}\|g\| +[\zeta]_{\alpha-2,B(z,kR)}\nonumber\\
&+(kR)^{-\alpha}\|u\|_{ B(z,kR)}.
\end{align}
We prove this Schauder estimate by applying some cut-off functions and using the Lemma~\ref{lemshauder}.
By scaling and translation, it is enough to prove for some $C^\alpha$ function $U$,
\begin{align}\label{schauder scaled}
[U]_{\alpha,B(0,\frac12)}\lesssim &\sup_{T\leqslant1}T^{2-\alpha}\|\heat (U\ind{B(0,1)})_T\|+\|U\|_{ B(0,1)}.
\end{align}
Indeed, since $[u]_{\alpha,B(0,\frac1k)}\leqslant[u]_{\alpha,B(0,\frac12)}$, if we have \eqref{schauder scaled}, define $U(t,x)=u((kR)^2(t-t_0),kR(x-x_0))$. Then 
\[
\|u\|_{ B(z,kR)}=\|U\|_{ B(0,1)},\quad [u]_{\alpha,B(z,R)}=(kR)^\alpha[U]_{\alpha,B(0,\frac1k)}
\]
\[
\sup_{T\leqslant kR}T^{2-\alpha}\|\heat (u\ind{B(z,kR)})_T\|=\sup_{T\leqslant1}T^{2-\alpha}\|\heat (U\ind{B(0,1)})_T\|.
\] 
We proceed to prove \eqref{schauder scaled}.
Let $\eta$ be a cut-off function, with value $1$ on $B(0,\frac12)$ and $0$ on $B(0,1)^C$, and such that $\|\triangledown\eta\| \leqslant 4$ and $\|(\Delta+\partial_t)\eta\| \leqslant 4$. Then
\begin{equation}\label{cut off e}
(\partial_t-\Delta)U\eta=\eta\heat U+U(\partial_t+\Delta)\eta-2\triangledown.(U\triangledown\eta ).
\end{equation}
By applying Lemma~\ref{lemshauder} to Equation \eqref{cut off e} we get that:
\begin{equation}\label{lemma 9}
[U\eta]_{\alpha}\lesssim\sup_{0<T<1}T^{2-\alpha}\|(\eta\heat U+U(\partial_t+\Delta)\eta-2\triangledown.(U\triangledown\eta ))_T\| .
\end{equation}
We apply the triangle inequality and make use of \eqref{e:etazeta} to bound each of these terms as follows.
\begin{align*}
\|(\eta\heat U)_T\|\lesssim \|\heat( U\ind{(B(0,1)})_T\|,
\end{align*}
\begin{align*}
\|(U(\partial_t+\Delta)\eta)_T\|\leqslant\|U(\partial_t+\Delta)\eta\|\lesssim \|U\|_{B(0,1)},
\end{align*}
\begin{align*}
\|(\triangledown.(U\triangledown\eta))_T\| =&\sup_z\int(U\triangledown\eta)(z-\bar{z}).\triangledown \Psi_T(\bar{z})d\bar{z}\\
\leqslant &\|U\triangledown\eta\| \|\triangledown \Psi_T\|_{L^1}\\
\lesssim &\inv{T}\|U\|_{ B(0,1)}.
\end{align*}
Since $\alpha<1$, we have
\begin{align*}
\sup_{0<T<1}T^{2-\alpha}\|(\eta\heat U+U(\partial_t+\Delta)\eta-2\triangledown.(U\triangledown\eta ))_T\|\\
\lesssim \|\heat( U\ind{(B(0,1)})_T\|+\|U\|_{ B(0,1)}.
\end{align*}
This concludes the proof of \eqref{schauder scaled}, hence the proof of \eqref{schauder}.

\bigskip

\paragraph*{\textsc{Step} 2: Application of Maximum principle}
We convolve the equation \eqref{phi41} with $\Psi_T$, where $T\in (0,1)$ will be specified later:
\begin{equation}\label{phi41_T}
(\partial_t-\Delta)u_T=-(u_T)^m+g_T+\zeta_T+\left((u_T)^m-(u^m)_T\right).
\end{equation}
 Theorem~\ref{main thm poly} implies that for all $0<R'<R<\inv{2}$,

\begin{align}\label{apply caccio}
\|u_T\|_{P_R}\lesssim\max\Big{\{}&\left(\frac1{(R-R')^2}\right)^\inv{m-1},\|g\| ^\frac1m,\|\zeta_T\|_{P_{R'}}^\inv{m},\nonumber\\
&\left(\|(u_T)^m-(u^m)_T\|_{P_{R'}}\right)^\inv{m}\Big{\}}.
\end{align}
The goal is now to balance the commutator and the term with the noise. This will be done by choosing the parameter $T$ appropriately.

\bigskip

\paragraph*{\textsc{Step} 3: Bounds on the commutator}
 We need estimates on the commutator $(u_T)^m-(u^m)_T$. This is obtained as $u$ is $C^\alpha$, using the moment bounds \eqref{moment of psi} and \eqref{mollified regularity}.
\begin{align*}
((u&_T)^m-(u^m)_T)(z)=\int\Psi_T(z-\bar{z})\left(u_T(z)^m-u(\bar{z})^m\right)d\bar{z}\\
=&\int\Psi_T(z-\bar{z})\int_0^1\left((u)_T(z)-u(\bar{z})\right)m\left(\lambda(u)_T(z)+(1-\lambda)u(\bar{z})\right)^{m-1}d\lambda d\bar{z}\\
\leqslant& m\|u\|_{B(z,T)}^{m-1}\int\Psi_T(z-\bar{z})\left(u_T(z)-u(z)+u(z)-u(\bar{z})\right)d\bar{z}\\
\leqslant& m\|u\|_{B(z,T)}^{m-1}\int\Psi_T(z-\bar{z})\left(T^\alpha[u]_{\alpha,B(z,T)}+[u]_{\alpha,B(z,T)}d(z,\bar{z})^\alpha\right)d\bar{z}\\
\leqslant& 2 m\|u\|_{B(z,T)}^{m-1}T^\alpha[u]_{\alpha,B(z,T)}.
\end{align*}
Since this is true for all $z\in P_R$,
\begin{equation}\label{commutator estimate 1}
\|(u_T)^m-(u^m)_T\|_{P_R}\leqslant 2m\|u\|_{P_{R-T}}^{m-1}\sup_{z\in P_R}[u]_{\alpha,B(z,T)}T^\alpha.
\end{equation}
Using the local Schauder estimate \eqref{schauder} gives, for any $k>2$:
\begin{align}\label{commutator estimate 2}
\|(u_T)^m-(u^m)_T\|&_{P_R}\lesssim T^2k^{2-\alpha}(\|u\|_{P_{R-kT}}^{2m-1}+\|u\|_{P_{R-kT}}^{m-1}\|g\| )\nonumber\\
&+\|u\|_{P_{R-T}}^{m-1}\sup_{z\in P_R}[\zeta]_{\alpha-2,B(z,kR)}T^\alpha+k^{-\alpha}\|u\|_{P_{R-kT}}^m.
\end{align}

\bigskip

\paragraph*{\textsc{Step} 4: Boot-strapping}
We show here that for $k,T$ such that $2(k+1)T\leqslant1$, with as before $k>2$, and for $1\geqslant R\geqslant 2(k+1)T$, we have
\begin{align}\label{big step 4}
\|u\|_{P_{R}}\lesssim\max\Big{\{}&R^\frac2{1-m},\|g\| ^\frac1m,\left([\zeta]_{\alpha-2,P_0} T^{\alpha-2}\right)^\frac1m,(T^2k^{2-\alpha})^\frac{1}{m}\|u\|_{P_0}^{2-\frac{1}{m}},\nonumber\\
 &\left(T^2k^{2-\alpha}\right)^\frac1m\|u\|_{P_0}^{1-\frac1m}\|g\| ^\frac1m,\left(\|u\|_{P_0}^{m-1}[\zeta]_{\alpha-2,P_0}T^\alpha\right)^\frac1m,k^{-\frac{\alpha}m}\|u\|_{P_0}\nonumber\\
& T^2 k^{2-\alpha}\|u\|^m_{P_0},T^2 k^{2-\alpha}\|g\| ,T^\alpha[\zeta]_{\alpha-2,P_0},k^{-\alpha}\|u\|_{P_0},\Big{\}}.
\end{align}
 We need to be careful with the sets that are concerned by the norms since our different estimates always require a bit more space. We use the bound \eqref{mollified regularity} and the Schauder estimate \eqref{schauder}:
\begin{align}\label{big step 3b}
\|u\|_{P_{R}}\leqslant&\|u_T\|_{P_{R}}+T^\alpha\sup_{z\in P_R}[u]_{\alpha,B(z,T)} 
\\\leqslant&\|u_T\|_{P_{R+T}}+T^2 k^{2-\alpha}(\|u\|^m_{P_{R-kT}}+\|g\| )\nonumber\\
&+T^\alpha[\zeta]_{\alpha-2,P_{R-kT}}+k^{-\alpha}\|u\|_{P_{R-kT}}.
\end{align}

Defining $r=kT$ allows to apply the bounds \eqref{apply caccio},\eqref{commutator estimate 2} and \eqref{shauder ou}. As $r\geqslant0$, $P_{r}\subset P_0$. 
\begin{align}\label{big step 3a}
\|u_T\|_{P_{R-T}}\lesssim& \max\Big{\{}(R-(k+1)T)^\frac2{1-m},\|g\| ^\frac1m,\|\zeta_T\|_{P_{r}}^\frac1m, \|(u_T)^m-(u^m)_T\|_{P_{r}}^\inv{m}\Big{\}}\nonumber\\
\lesssim \max\Big{\{}&(R-(k+1)T)^\frac2{1-m},\|g\| ^\frac1m,\left([\zeta]_{\alpha-2,P_0} T^{\alpha-2}\right)^\frac1m,(T^2k^{2-\alpha})^\frac1m\|u\|_{P_0}^{2-\frac1m},\nonumber\\
  &\left(T^2k^{2-\alpha}\right)^\frac1m\|u\|_{P_0}^{1-\frac1m}\|g\| ^\frac1m,\left(\|u\|_{P_0}^{m-1}[\zeta]_{\alpha-2,P_0}T^\alpha\right)^\frac1m,k^{-\frac{\alpha}{m}}\|u\|_{P_0}\Big{\}}.
\end{align}
If we start with  $R\geqslant 2(k+1)T$ then $R-(k+1)T\geqslant\frac{R}2$. Putting together \eqref{big step 3a} and \eqref{big step 3b} gives \eqref{big step 4}.

\bigskip

\paragraph*{\textsc{Step} 5: Choosing $T$}
In order to balance the term containing $\zeta$ in \eqref{big step 4}, we see that we should assign the value $T=\frac{\mu}{\|u\|_{P_0}^\frac{m-1}{2}}$ for some $\mu\in(0,1)$ to be chosen. 
Note also that as $\mu\in(0,1)$, $(\mu^{\alpha-2}\vee\mu^\alpha)=\mu^{\alpha-2}$. Furthermore, we impose $\mu^2k^{2-\alpha}\leqslant 1$. Consequently, \eqref{big step 4} becomes
 \begin{align}\label{BIG STEP}
\|u\|_{P_{R}}\lesssim\max&\Big{\{}R^\frac2{1-m},(1+(\mu^2k^{2-\alpha})^\frac1m)\|g\| ^\frac1m,\left(\mu^{\alpha-2}\|u\|^{(m-1)\frac{2-\alpha}2}_{P_0}[\zeta]_{\alpha-2,P_0}\right)^\frac1m,\nonumber\\
&((\mu^2k^{2-\alpha}\vee  k^{-\alpha})^\inv{m}+(\mu^2k^{2-\alpha}\vee k^{-\alpha}))\|u\|_{P_0},\nonumber\\
&\mu^2k^{2-\alpha}\|u\|_{P_0}^{1-m}\|g\|,\frac{\mu^\alpha}{\|u\|_{P_0}^{(m-1)\frac{\alpha}2}}[\zeta]_{\alpha-2,P_0} \Big{\}}.
\end{align}

\bigskip

\paragraph*{\textsc{Step} 6: Identification of terms}

We claim that the bound above implies that there exists a positive constants $C$ such that:
 \begin{equation}\label{Temporary conclusion}
\|u\|_{P_{R}}\leqslant C\max\Big{\{}R^\frac{2}{1-m},\|g\| ^\frac1m,[\zeta]_{\alpha-2,P_0}^\inv{1+(m-1)\frac\alpha2},\inv{2C}\|u\|_{P_0}\Big{\}}.
\end{equation}

 We need to interpolate some of the arguments of the maximum in \eqref{BIG STEP} with arguments of our goal \eqref{Temporary conclusion}. The first two terms are already in the right form. For the next one, a simple interpolation inequality gives that for any $\gamma>0$,
\[
\left(\mu^{\alpha-2}\|u\|^{(m-1)\frac{2-\alpha}2}_{P_0}[\zeta]_{\alpha-2,P_0}\right)^\frac1m\lesssim\mu^{\frac{\alpha-2}m}\max\Big{\{}\gamma\|u\|_{P_0},\gamma^\frac{\alpha-2}{\frac2{m-1}+\alpha}[\zeta]_{\alpha-2,P_0}^\inv{1+(m-1)\frac\alpha2}\Big{\}}.
\]
The next term is also in the right form, provided one chooses first $k$ large, and then $\mu$ small. The last two terms can not be dealt with with classical interpolation, since they involve negative powers of $\|u\|_{P_0}$.
 For the first one, we state that always one of the following is true, for any $\gamma>0$:
\[\|u\|^{-(m-1)}_{P_0}\|g\| \leqslant\gamma\|u\|_{P_0}\text{ or }\|u\|_{P_0}^{m}\leqslant\frac1\gamma\|g\| .\]
The first case gives the last argument of our objective for $\gamma$ small enough. The second case gives $\|u\|_{P_{R}}\leqslant\|u\|_{P_0}\leqslant(\frac1\gamma\|g\|)^\frac1m$. We proceed similarly for the last term. One of the following is always true:
\[\mu^\alpha\|u\|^{-(m-1)\frac{\alpha}{2}}_{P_0}[\zeta]_{\alpha-2,P_0}\leqslant\mu^\alpha\gamma\|u\|_{P_0}\text{ or }\|u\|_{P_0}^{1+(m-1)\frac{\alpha}{2}}\leqslant\inv{\gamma}[\zeta]_{\alpha-2,P_0}.\]
Once again the first case gives the last argument of our objective for $\mu^\alpha\gamma$ small enough, and the second case gives $\|u\|_{P_{R}}\leqslant\|u\|_{P_0}\leqslant(\frac1\gamma[\zeta]_{\alpha-2,P_0})^\inv{1+(m-1)\frac\alpha2}$.
 We can then choose $k$ large, $\mu$ and $\gamma$ small to get the desired constant $C$.
\bigskip

\paragraph*{\textsc{Step} 7: Iterating the result}
The last argument of the maximum \eqref{Temporary conclusion} is greater than the first one for all $R$ such that
 \[R\leqslant R_1:=\left(\inv{2C}\|u\|_{P_0}\right)^\frac{1-m}{2},\]

Let us check that this is not in contradiction with $R_1\geqslant 2(k+1)T$. By defintion of $T$ and $R_1$, 
\[ 2(k+1)T=2(k+1)\frac{\mu}{\|u\|_{P_0}^\frac{m-1}{2}}=2(k+1)\mu(2C)^\frac{1-m}{2}R_1\leqslant R_1\Leftrightarrow 2(k+1)\mu(2C)^\frac{1-m}{2}\leqslant 1.\]
Since $C>1$ and $m>1$, it is enough to have  $2(k+1)\mu\leqslant 1$. This can be done since $\mu$ is chosen after $k$.

From this point, the result \eqref{Temporary conclusion} can be iterated to get bounds for smaller and smaller parabolic boxes.
\[\|u\|_{P_{(R+R_{n-1})}}\leqslant C\max\Big{\{}\inv{2C}\|u\|_{P_{R_{n-1}}},\|g\| ^\frac1m,[\zeta]_{\alpha-2,P_0}^\inv{1+(m-1)\frac{\alpha}{2}},\inv{R^\frac{2}{m-1}}\Big{\}}.\]
Define $R_n$ recursively by
\begin{equation}\label{defrn}
R_n-R_{n-1}=\left(\inv{2C}\|u\|_{P_{R_{n-1}}}\right)^\frac{1-m}2=\left(\inv{2C}\frac{\|u\|_{P_0}}{2^{n-1}}\right)^\frac{1-m}2.
\end{equation}

We conclude by summing those increments:
\begin{align}\label{studyR}
R_n&=\sum_{k=1}^{n}R_k-R_{k-1}=\sum_{k=1}^{n}\left(\inv{2C}\frac{\|u\|_{P_0}}{2^{k-1}}\right)^\frac{1-m}2.\nonumber\\
=&\left(\frac{\|u\|_{P_0}}{2C}\right)^\frac{1-m}2\sum_{k=0}^{n-1}(2^\frac{1-m)}2)^k\lesssim\left(\frac{\|u\|_{P_0}}{2C}\right)^\frac{1-m}2
\end{align}
The same arguments as in the proof of Lemma~\ref{lem:ODE} conclude the proof of Theorem~\ref{The Theorem}.

\section{Mutiplicative noise}\label{sec:multipliative}
We present one example of equation where our result applies. Let $(\Omega,\mathcal{F},\mathcal{F}_t,\mathbb{P})$ be a filtered probability space and let 
$(W(t,\eta), t \geq 0, \eta \in C^{\infty}_0(\R^d))$ be a Brownian motion
with spatial covariance operator $K$ on $\Omega$. We assume that $K$ is given by the convolution with a function with controlled blow-up near the origin, i.e. 

\begin{equation}\label{K-def}
K \phi (x) = \int_{\R^d} K(x-x') \phi(x') dx',
\end{equation}
for $K \in C^{\infty}(\R^d \setminus \{ 0 \})$ satisfying 
\begin{align}
|K(x)| \leq \frac{1}{|x|^{\lambda}},
\end{align}
for some $\lambda <2$. If $\lambda >1$ and $d=1$, we allow additionally for a Dirac mass in the origin, in which case \eqref{K-def} turns into 
\begin{align}\label{K-def-bis}
K \phi (x) = \int_{\R} K(x-x') \phi(x') dx' + \phi(x).
\end{align}
In other words $(W(t,\eta), t \geq 0, \eta \in C^{\infty}_0(\R^d))$ is a centred Gaussian process with covariances given 
either by 
\[
\E W(t, \phi) W(t', \phi') = (t \wedge t' ) \int_{\R^d}\int_{\R^d} \phi(x) K(x-x') \phi' (x')  dx dx' 
\]
or in the one-dimensional case 
\begin{align*}
&\E W(t, \phi) W(t', \phi') \\
& \quad = (t \wedge t' )  \Big[ \int_{\R^d}\int_{\R^d} \phi(x) K(x-x') \phi' (x')  dx dx'  + \int_{\R^d} \phi(x) \phi'(x) dx \Big].
\end{align*}
Let $(\sigma(t,x), t \geq 0, x \in \R^d)$ be a progressively measurable process, 
with a deterministic $L^\infty$ bound, without loss of generality $|\sigma(t,x)| \leq 1$. 
Let $u(t,x)$ be a continuous process which satisfies the SPDE 
\begin{equation}\label{phi multnoise}
du=(\Delta u-f(u)+g(u))dt+\sigma dW
\end{equation} 
on $P_0$, with $f$ satisfying the Assumptions~\ref{ass2}. More precisely, for all $\eta \in C^{\infty}(\R \times \R^d)$ compactly supported in $P_0$ we assume that the following holds almost surely:
\begin{align}\notag
 &\int \int u (- \partial_t -\Delta) \eta dt dx \\
 \label{phi_multnoisebis}
 & \qquad = \int \int (- f(u,z) + g(u,z))  \eta dt dx+ \int \int \eta(x)  \sigma(t,x) dx  dW(t,x),
\end{align}
where $ \int \int \eta(x)  \sigma(t,x) dx  dW(t,x)$ should be interpreted as a stochastic integral, as defined in \cite[Chapter 4]{da2014stochastic}.
The following Lemma shows that the results of our deterministic analysis are applicable to this stochastic case.

The previous results do not depend on the particular choice of convolution kernel $\Psi$. We apply it with $\tilde{\Psi}$ defined as
\begin{equation}\label{def kernel}
\tilde{\Psi}=\Psi_\frac12\ast\Psi_\frac12,
\end{equation}
where $\Psi$ is as defined in Section \ref{sec:Main}. It is clear that  $\tilde{\Psi}$ is still non-negative, smooth and compactly supported in $B(0,1)$.
We still write $(\cdot)_T$ for the convolution with $\Psi_T$ but we define the $C^{\alpha-2}$ norm with respect to $\tilde{\Psi}$
\[
[\zeta]_{\alpha-2,C}=\sup_{0<T\leqslant 1}\|(\zeta_\frac{T}2)_\frac{T}2\|_CT^{2-\alpha}
\]
%
\begin{lem}\label{lem mult}
We define a family of random variables $(\zeta(\eta), \eta \in C^{\infty}_0(\R \times \R^d))$ by
\[
\zeta(\eta) = \int \int \eta(x)  \sigma(t,x) dx  dW(t,x).
\]
Then there exists a random distribution $\tilde{\zeta}$ on $\Omega$ which almost surely  takes values in  $C^{\alpha-2}$ for any $\alpha < \frac{2-\lambda}{2}$
and such that for $\eps>0$ small enough
\begin{equation}\label{multnoise bound}
\Exp{\exp \Big(\epsilon [\tilde\zeta]_{\alpha-2,P_0}^2\Big)} <
\infty.
\end{equation}
Furthermore $\tilde{\zeta}$ is a modification of $\zeta$ in the sense that for all $\eta \in C^{\infty}_0(\R \times \R^d))$ we have almost surely
\[
\tilde\zeta(\eta) = \zeta(\eta).
\]
\end{lem}
The following corollaries are consequences of Lemma~\ref{lem mult}. Using first Corollary~\ref{cor:dpd}, as well as Lemma \ref{lemshauder}
which provides bound on the $\alpha$ H\"older semi-norm of $w$ in terms of $[\zeta]_{\alpha-2}$, which in turn controls the 
supremum norm using the Dirichlet boundary conditions, we get:
\begin{cor}\label{SPDE_moment_bound_gen}
Let $u$ solve the SPDE \eqref{phi multnoise} in the sense of \eqref{phi_multnoisebis} for $f$ and $g$ satisfying Assumption~\ref{ass2}. Define $\Theta(u)=\frac{f_1(u)}{u}$. Then 
there exists $\eps_0 = \eps_0( c,d,\alpha)>0$ such that
for $0< \eps \leqslant \eps_0$,
\begin{align*}
\E \bigg[\exp \bigg(\eps \Big(\sup_{0< R \leq \frac12} \frac{ \| u \|_{P_R}}{\Theta^{-1}((\lambda R)^{-2})} \Big)^2 \bigg)  \bigg] < \infty.
\end{align*}
\end{cor}
Using Theorem~\ref{The Theorem} in the case $f(u)=u|u|^{m-1}$, we have the more optimal estimate as follows:
\begin{cor}\label{SPDE_moment_bound}
Let $u$ solve the SPDE \eqref{phi multnoise} in the sense of \eqref{phi_multnoisebis} where $f(u,z)=u|u|^{m-1}$ and $g$ is bounded. Then there exists $\eps_0=\eps_0(m,d,\alpha)>0$
such that
for $0< \eps \leqslant \eps_0$,
\begin{align*}
\E \bigg[\exp \bigg(\eps \Big(\sup_{0< R \leq \frac12} R^{\frac{2}{m-1}} \| u \|_{P_R} \Big)^{2+(m-1)\alpha} \bigg)  \bigg] < \infty.
\end{align*}

\end{cor}

The proof of Lemma~\ref{lem mult} relies on the following technical lemma.
\begin{lem}\label{lem mult noise}
The supremum $\sup_{0< T\leqslant 1}\|(\zeta)_T\|_{P_0}^{2p}T^{2p(2-\alpha)}$ is bounded by the supremum over dyadic $T$ only,
\begin{equation}\label{Dyadic}
\sup_{0< T\leqslant 1}\|\zeta_T\|_{P_0}^{2p}T^{2p(2-\alpha)}\lesssim \sup_{T=2^{-k}\leqslant 1}\|\zeta_T\|_{P_0+B(0,1)}^{2p}T^{2p(2-\alpha)}.
\end{equation} 
\end{lem}
We give the proof of this in Appendix~\ref{appendix 1}.
\begin{proof}[Proof of Lemma~\ref{lem mult}] 
This Lemma is a variant of \cite[Lemma 9]{mourrat2017global} and we refer the reader to this Lemma for the construction of 
a suitable modification of $\zeta$. Here we only show the exponential integrability bound \eqref{multnoise bound}, using a similar 
argument as in \cite[Lemma 4.1]{OW}.
Throughout this proof, $\lesssim$ denotes a bound up to a constant that depend only on the dimension.

In the expansion in series of the exponential, we can exchange expectation and sum:
\[\E\Big[\exp\Big(\epsilon^2\sup_{0<T\leqslant 1}\|(\zeta_\frac{T}2)_\frac{T}2\|_{P_0}^2T^{2(2-\alpha)}\Big)\Big]=\sum_{p=0}^\infty\epsilon^{2p}\frac{\Exp{\sup_{0\leqslant T\leqslant 1}\|(\zeta_\frac{T}2)_\frac{T}2\|_{P_0}^{2p}T^{2p(2-\alpha)}}}{p!}.\]
Applying Lemma~\ref{lem mult noise}, we can bound the supremum over all $T$ by the sum over dyadic $T$.
\[\Exp{\sup_{0<T\leqslant 1}\|(\zeta_\frac{T}2)_\frac{T}2\|_{P_0}^{2p}T^{2p(2-\alpha)}}\leqslant\sum_{T=2^{-k}\leqslant \frac12}\Exp{\|(\zeta_T)_T\|_{P_0+B(0,1)}^{2p}}T^{2p(2-\alpha)}.\]
Young's inequality implies 
\[
\|(\zeta_T)_T\|_{P_0+B(0,1)}\leqslant\|\zeta_T\|_{L^q,P_0+B(0,2)}\|\Psi_T\|_{L^{q'}},
\]
where the subscript means that the $L^q$ norm of $\zeta_T$ is taken over $P_0+B(0,2)$ and
where $q'=\frac{q-1}{q}$. By scaling, $\|\Psi_T\|_{L^{q'}}\lesssim T^{-\frac{d+2}{q}}$. We apply this with $q=2p$.
\begin{align*}
\Exp{\|(\zeta_T)_T\|^{2p}}&\lesssim\Exp{\|\zeta_T\|_{L^{2p},P_0+B(0,2)}^{2p}}T^{-(d+2)}\\
&=\Exp{\int_{P_0+B(0,2)}\zeta_T(t,x)^{2p}dtdx}T^{-(d+2)}\\
&\lesssim \sup_{z\in P_0+B(0,2)}\Exp{\zeta_T(z)^{2p}}T^{-(d+2)}.
\end{align*}
We bound $\Exp{\zeta_T(z)^{2p}}$ using the boundedness of $\sigma$. Without loss of generality, we show the computation for $z=(0,0)$.
 By the Burkholder-Davies-Gundy inequality, 
\begin{align*}
\Exp{\zeta_T(0,0)^{2p}}&=\Exp{\Big(\int_{(0,1)}\int_{\R^d}\Psi_T(t,x)\sigma(t,x)dW(t,x)\Big)^{2p}}\\
&\lesssim p^p\Big{(}\int\int\int\Psi_T(t,x)\Psi_T(t,x') \sigma(t,x)\sigma(t,x')K(x-x')dtdxdx'\Big{)}^p\\
&+\ind{d=1,\lambda>1}p^p\Big{(}\int\int \Psi_T(t,x)^2 \sigma(t,x)^2dtdx\Big{)}^p\\
&\lesssim p^p ( T^{-\lambda -2}+\ind{d=1,\lambda>1}T^{-d-2})^p\lesssim p^pT^{-p(\lambda +2)}.
\end{align*}
We get that $\Exp{\|(\zeta)_T\|_{}^{2p}}T^{2p(2-\alpha)}\lesssim p^{p}T^{p(2-2\alpha-\lambda)-(d+2)}$. Since $2-2\alpha-\lambda>0$, for $p$ large enough,
\[
\sum_{T=2^{-k}\leqslant \frac12}\Exp{\|(\zeta_T)_T\|^{2p}}T^{2p(2-\alpha)}\lesssim p^{p}\frac{1}{1-2^{-p(2-2\alpha-\lambda)+(d+2)}}.
\]
By Stirling's formula, for $p$ large enough,
\[\epsilon^{2p}\frac{\Exp{\sup_{0\leqslant T\leqslant 1}\|(\zeta_\frac{T}2)_\frac{T}2\|_{}^{2p}T^{2p(2-\alpha)}}}{p!}\lesssim \epsilon^{2p}e^p\sqrt{p},\]
hence for $\epsilon<e^{-2}$, \eqref{multnoise bound} is verified.
\end{proof}
\section{Invariant measure and Optimality}\label{sec:Optimality}
In this last section, we consider a special case of the SPDE considered in Section~\ref{sec:multipliative}, namely the 
case of a one-dimensional reaction-diffusion equation driven by an additive space-time white noise. We aim to argue that 
in this case the bound obtained in Corollary~\ref{SPDE_moment_bound} is optimal in terms of stochastic integrability.

Let $d=1$ and let $W$ be as in Section~\ref{sec:multipliative} with covariance operator $K\eta(x) = \eta(x)$. It is well-known 
\cite[Section 11.2]{da1996ergodicity} that if we impose Dirichlet boundary conditions on the space-interval $[-1,1]$, then 
\eqref{phi_multnoisebis} defines a reversible Markov process with respect to the measure
%
%
%
%
%
\begin{equation}\label{Invariant measure}
\inv{Z}\exp\left(-\int_{-1}^1\inv{m+1} |u(x)|^{m+1}dx\right)\mu(du),
\end{equation}
where $\mu$ is the law of an appropriately scaled Brownian bridge and $Z$ is a normalisation constant. 
From the explicit expression \eqref{Invariant measure} one can immediately read of that under this measure the following expectations are finite
for $\alpha<\frac12$ and $\epsilon$ small enough
\begin{equation}\label{bounded expectation}
\Exp{\exp \Big(\epsilon\int_{-1}^1 |u|^{m+1} dx \Big)}<\infty\text{ and }
\Exp{\exp\left(\epsilon [u]_{\alpha}^2\right)}<\infty.
\end{equation}
The following proposition shows how to interpolate these two estimates to get optimal stochastic integrability for 
 the supremum norm $\|u\|$. 

\begin{prop}\label{prop interpolation}
 If $u\in C^\alpha(-1,1)$ and $u^{m+1}$ is integrable, $u$ is bounded and we have the following interpolation:
\begin{equation}\label{alpha p interpolation}
\Big(\frac{\|u\|_{(-1,1)}}{2}\Big)^{1+\alpha(m+1)}\leqslant \max\{[u]_{\alpha}\|u\|_{m+1}^{\alpha(m+1)},\|u\|_{m+1}^{1+\alpha(m+1)}\},
\end{equation}
where $\|.\|_{m+1}$ refers to the $L^{m+1}$ norm on $[-1,1]$.
\end{prop}
Since $2\alpha<1$, 
\[\| u\|_{m+1}^{(m+1)\alpha}[u]_{\alpha}\leqslant\| u\|_{m+1}^{m+1}+[u]_{\alpha}^2.\]
Hence, \eqref{bounded expectation} implies that for $\epsilon$ small, 
\begin{equation}\label{optim1}
\Exp{\exp\left(\epsilon\|u\|^{1+(m+1)\alpha}\right)}<\infty.
\end{equation}
On the other hand, from Theorem~\ref{The Theorem} and from Corollary~\ref{SPDE_moment_bound}, we get
\begin{equation}\label{optim2}
\E \Big[\exp \Big( \eps(2^{-\frac{2}{m-1}} \| u \|_{P_\frac12})^{2+(m-1)\alpha} \Big)  \Big] <\infty.
\end{equation}
Therefore, for $\alpha\rightarrow \inv{2}$, the exponents in \eqref{optim1} and \eqref{optim2} both converge to $\frac{m+3}{2}$. For one dimensional space-time white noise, the term $[\zeta]_{\alpha-2,P_0}^\inv{1+(m-1)\frac\alpha2}$ in Theorem~\ref{The Theorem} is an optimal control.

\appendix
\section{Proof of lemma~\ref{lem mult noise}}\label{appendix 1}
This proof essentially follows \cite[Lemma A.3]{OW}.
 By splitting the interval $(0,1)$ into $[2^{-n},2^{-n+1})$ for $n\geqslant 1$, it is enough to prove that for all $n\geqslant1$, uniformly in $\lambda\in(0,1)$

\begin{equation}
\|(\zeta)_{2^{-n}(1+\lambda)}\|\lesssim 2^{-n(\alpha-2)}\sup_{T=2^{-k}\leqslant1}T^{2-\alpha}\|\zeta_T\|.
\end{equation}
For $m\in\mathbb{N}^{d+1}$, denote $|m|=2m_0+\sum_{k=1}^dm_k$ the parabolic index, and $m!=\prod_{k=0}^d m_k!$.
We define $A_{n,m}=\int z^n\partial^m\Psi(z)dz$ and observe that since $\Psi $ is compactly supported and 
since $\int\Psi=1$, we have
\begin{align*}
A_{n,m} 
=&\left\{\begin{array}{cc}
0 & \text{if } |n|\leqslant |m|, n\neq m, \\ 
(-1)^{|m|}m! & \text{if } n=m.
\end{array} \right.
\end{align*}
Hence for any $\beta>0$, $(A_{n,m})_{|n|,|m|\leqslant \beta}$ is an invertible linear system. By continuity of the coefficients, for $r$ small enough (depending on $\beta$),
\begin{equation}
A_{n,m}^r=\int z^n\partial^m(\Psi_r\ast\Psi)(z)dz
\end{equation}
is also an invertible linear system. Hence, there exists coefficients $(a_m)_{|m|\leqslant \beta}$ such that 
\begin{equation}
\sum_{|m|\leqslant \beta}a_m\int z^n\partial^m(\Psi_r\ast\Psi)(z)dz=\left\{\begin{array}{cc}
1 & \text{if } n=0 \\ 
0 & \text{else. } 
\end{array} \right.
\end{equation}
Set $\omega^0=\sum_{|m|\leqslant \beta}a_m\partial^m\Psi_r$ and $\Psi'=\omega^0\ast \Psi$, then 
\begin{equation}
\int z^n\Psi'(z)dz=\left\{\begin{array}{cc}
1 & \text{if } n=0 \\ 
0 & \text{for } 0<|n|<\beta.
\end{array} \right.
\end{equation}
Define now for $\theta>0$ inductively $\omega^{(0,\lambda)}=\Psi_{1+\lambda}$ and $\omega^{(k+1,\lambda)}=\omega^{(k,\lambda)}-\Psi'_{\theta^k}\ast\omega^{(k,\lambda)}$, in order to get 
\begin{equation}\label{dyadic decomposition}
\Psi_{1+\lambda}=\sum_{k=0}^\infty\Psi'_{\theta^k}\ast\omega^{(k,\lambda)}=\sum_{k=0}^\infty\Psi_{\theta^k}\ast\omega^{(k,\lambda)}\ast\omega^0.
\end{equation}
Similar to \eqref{mollified regularity} on can see that $\int|\omega^{(k,\lambda)}-\Psi'_{\theta^k}\ast\omega^{(k,\lambda)}|\lesssim\theta^{(\beta+1) k}\int|\partial^\beta\omega{(k,\lambda)}|$ and $\int|\partial^\beta(\omega^{(k,\lambda)}-\Psi'_{\theta^k}\ast\omega^{(k,\lambda)})|\lesssim\int|\partial^\beta\omega^{(k,\lambda)}|$, hence by induction 
\begin{equation}
\int|\omega^{(k,\lambda)}|\lesssim \theta^{(\beta+1) k}\int|\partial^\beta\Psi_{1+\lambda}|\lesssim \theta^{(\beta+1) k}\sup_{\lambda\in(0,1)}\int|\partial^\beta\Psi_{1+\lambda}|.
\end{equation}
We can rescale \eqref{dyadic decomposition} by $T=2^{-k}$, and we have
\begin{align*}
\|(\zeta)_{2^{-n}(1+\lambda)}\|_{P_0}&=\|\sum_{k=0}^\infty\Psi_{\theta^k{2^{-n}}}\ast \omega^{(k,\lambda)}_{2^{-n}}\ast\omega^0_{2^{-n}}\ast \zeta\|_{P_0} \\
&\leqslant\sum_{k=0}^\infty\|(\zeta)_{\theta^k{2^{-n}}}\|_{P_0 + B(0,2^{-n})}\int|\omega^{(k,\lambda)}_{2^{-n}}|\int|\omega^{0}_{2^{-n}}|\\
&\leqslant\sup_{T=2^{-k}\leqslant1}T^{2-\alpha}\|\zeta_T\|_{P_0 + B(0,1)} \sup_{\lambda\in(0,1)}\int|\partial^\beta\Psi_{1+\lambda}|\int|\omega^{0}|\\
&\quad\quad\quad\sum_{k=0}^\infty(\theta^k{2^{-n}})^{\alpha-2}\theta^{(\beta+1) k}\\
&\lesssim2^{-n(\alpha-2)}\sup_{T=2^{-k}\leqslant1}T^{2-\alpha}\|\zeta_T\|_{P_0 + B(0,1)}.
\end{align*}

\section{Proof of proposition~\ref{prop interpolation}}.

For any interval $I\subset [-1,1]$, 
\[|u(t)|-\inv{|I|}\Big|\int_Iu(s)ds\Big|\leqslant\inv{|I|}\int_I|u(t)-u(s)|ds\leqslant\inv{|I|}\int_I[u]_{\alpha,I}|t-s|^\alpha.\]
If $t\in I$, $|t-s|^\alpha\leqslant I^\alpha$. We can apply Jensen's inequality. 
\[|u(t)|\leqslant \left(\inv{|I|}\int_I|u(s)|^{m+1}ds\right)^\frac1{m+1}+|I|^\alpha[u]_{\alpha,I}.\]
And since this is true for any $I\subset [-1,1]$, we have for any choice of $0<x\leqslant 2$,
\[\|u\|_{J}\leqslant x^{-\frac1{m+1}}\|u\|_{m+1}+x^\alpha[u]_{\alpha,J}.\]
 If $\|u\|_{m+1}\geqslant [u]_{\alpha}$ then choose $x=1$ to get $\|u\|_{(-1,1)}\leqslant 2\|u\|_{m+1}$. Else choose $x=(\|u\|_{m+1}/[u]_{\alpha})^\frac{m+1}{1+\alpha(m+1)}\leqslant 1$ and get $\|u\|_{(-1,1)}\leqslant 2[u]_{\alpha}^\frac1{1+\alpha(m+1)}\|u\|_{m+1}^\frac{\alpha(m+1)}{1+\alpha(m+1)}$. In conclusion, 
\[\|u\|_{(-1,1)}\leqslant 2\max\{\|u\|_{m+1},[u]_{\alpha}^\frac1{1+\alpha(m+1)}\|u\|_{m+1}^\frac{\alpha(m+1)}{1+\alpha(m+1)}\}.\]
\bibliographystyle{abbrv}
\bibliography{phip}

%
%
%
%
%
%
%
%
%
%
%
%

\end{document}